\documentclass[a4paper]{amsart}
\usepackage{amsmath,amssymb,amsthm}
\newtheorem{thm}{Theorem}

\newtheorem{cor}[thm]{Corollary}
\newtheorem{prop}[thm]{Proposition}

\newtheorem{lem}[thm]{Lemma}
\newtheorem{remark}{Notation}

\newcommand{\sX}{\mathfrak X}

\newcommand{\sA}{\mathfrak A}

\newcommand{\sC}{\mathfrak C}
\newcommand{\sD}{\mathfrak D}

\newcommand{\sF}{\mathbb F}

\newcommand{\qbinom}[2]{\binom{#1}{#2}_q}
\newcommand{\qqbinom}[2]{\binom{#1}{#2}_{q^2}}

\title[Modular dual polar schemes]{Modular adjacency algebras of Dual polar schemes}
\author{Osamu Shimabukuro}
\address[O. Shimabukuro]{Department of Mathematics, Faculty of Education,
 Nagasaki University, 1-14 Bunkyo-machi, Nagasaki 852-8521, Japan}
\email{shimabukuro@nagasaki-u.ac.jp}
\date{June 10, 2015}
\begin{document}

\maketitle

\begin{abstract}
We can define the adjacency algebra of an association scheme over arbitrary field.
It is not always semisimple over a field of positive characteristic.
The structures of adjacency algebras over a field of positive characteristic have not been sufficiently studied.

In this paper, we consider the structures of adjacency algebras of dual polar schemes over a field of positive characteristic and determine them when their algebras are local algebras.
\end{abstract}
\section{Introduction}
An adjacency algebra of an association scheme over a field of characteristic zero is called the Bose-Mesner algebra.
The Bose-Mesner algebras is always semisimple.
Many researchers have studied this case and there are many results \cite{bi,bcn}.
An adjacency algebra of an association scheme over a field of positive characteristic is called a modular adjacency algebra.
Hanaki and Yoshikawa determined the structure of the modular adjacency algebras and the modular standard modules of association schemes of class $2$ \cite{HY05}.
Using modular standard modules, they provided more detailed classification than using parameters of strongly regular graphs.
This indicates that structures of the modular standard modules of association schemes provide more detailed characterization than parameters of association schemes.
In order to determine the structure of the modular standard modules, we first need to obtain the structure of the modular adjacency algebras.
However, the structures of adjacency algebras over fields of positive characteristic have not been sufficiently studied \cite{H02,HY05,S07,S11,SY15,Y04}. 

In this paper, we consider the structure of modular adjacency algebras of all types of dual polar schemes
and determine the structure of the modular adjacency algebras of these schemes when their algebras are local algebras.
One of our main theorems is the following.  
\begin{thm}\label{MainTheorem1}
Let $F$ be a field of odd characteristic $p$, $r$ be an odd prime power which is not divided by $p$ and
$\sA_d$ be a dual polar scheme on $[{}^2 A_{2d}(r)]$.
If $F \sA_d$ is a local algebra,
then
$$
F \sA_d \cong P/W_d,
$$
where $P$ is a polynomial ring $F[X_1,X_2\dots]/(X_1^p,X_2^p,\dots)$  and $W_d$ is an ideal generated by monomials of $P$ such that their weights are grater then $d$.  
\end{thm}
These structures of modular adjacency algebras of  dual polar schemes on $[{}^2 A_{2d}(r)]$ decide other structures.
\begin{thm}\label{MainTheorem2}
Let $F$ be a field of odd characteristic $p$, $q$ be an odd prime power which is not divided by $p$ and
$\sC_{2d'+1}$  be a dual polar scheme on $[C_{2d'+1}(q)]$.
If $F \sC_{2d'+1}$ is a local algebra,
then
$$
F \sC_{2d'+1} \cong  P/W_{d'} \otimes P/W_{1}.
$$
\end{thm}
For dual polar schemes on the rest of types, we determined the structures of their modular adjacency algebras using isomorphisms or  epimorphisms.

The modular adjacency algebra of Hamming scheme $H(n,p)$ over a field of characteristic $p$ is also local algebra \cite{Y04},
because the modular adjacency algebra of an association scheme with a prime power order is a local algebra over a field of characteristic the prime \cite{H02}.
However the orders of our dual polar graphs are not prime powers.
Therefore  our situation is different from  the modular adjacency algebra of Hamming scheme.
\section{Preliminaries}
Let $X$ be a finite set with cardinality $n$.
We define 
$R_{0} ,\dots$, $R_d$ as symmetric binary relation on $X$.
The $i$-th adjacency matrix $A_i$ is defined to be the matrix indexed by $X$ whose entries are
\[
( A_{i} )_{x y} = \begin{cases} 1 & \text{ if } (x,y) \in R_{i}, \\ 0 & \text{ otherwise}. \end{cases} 
\]

A configuration $\sX = (X, \{R_i\}_{i=0}^d)$ is called a symmetric association scheme of class $d$ if 
$\{A_i\}_{i=0}^d$ satisfies the following:
\begin{enumerate}
\item~$A_0=I$ (identity matrix),
\item ~$\sum_{i=0}^d A_i = J$ (all-ones matrix),
\item~there exist real numbers $p_{i,j}^k$ such that 
\[A_i A_j = \sum_{k=0}^d p_{i,j}^k A_k\]
for all $i,j  \in \{0,1,...,d\}$ \label{def3}.
\end{enumerate}
The set of linear combination of all adjacency matrices is closed under multiplication,
because of the equation (\ref{def3}).
For a commutative ring $R$ with identity, we define $R \sX = \bigoplus_{i = 0}^{~d}~R A_{i}$ as a matrix ring over $R$;
this is called an adjacency algebra of $\sX$ over $R$.
A modular adjacency algebra $F \sX$ is an adjacency algebra of $\sX$ over a field $F$ of positive characteristic $p$.
\section{Dual polar schemes}
Let $q$ and $r$ be prime powers, and
$V$ be one of the following spaces equipped  with a specified form:
\begin{align*}
&[C_d(q)]=\sF^{2d}_{q} \text{ with a non-degenerate symplectic form;}\\
&[B_d(q)]=\sF^{2d+1}_{q} \text{ with a non-degenerate quadratic form;}\\
&[D_d(q)]=\sF^{2d}_{q} \text{ with a non-degenerate quadratic form of (maximal) Witt index $n$;}\\
&[^2 D_{d+1}(q)]=\sF^{2d+2}_{q} \text{ with a non-degenerate quadratic form of (non-maximal) Witt index $n$;} \\
&[^2 A_{2d}(r)]=\sF^{2d+1}_{q} \text{ with a non-degenerate Hermitian form $(q=r^2)$;} \\
&[^2 A_{2d-1}(r)]=\sF^{2d}_{q} \text{ with a non-degenerate Hermitian form $(q=r^2)$.} 
\end{align*} 

Let $X$ be the set of maximal totally isotropic subspaces of $V$. Each element of $X$ has dimension $d$. Define the $i$-th relation $R_{i}$ on $X$ by 
$$
(x,y) \in R_{i} \Longleftrightarrow \text{dim}(x \cap y)=d-i.
$$

Then $\sX=(X,\{R_{i}\}_{i=0}^{d})$ is a ($P$ and $Q$) polynomial scheme \cite{bi,D1};
this is called the dual polar scheme on $[C_d(q)]$, $[B_d(q)]$, $[D_d(q)]$, $[{}^2 D_{d+1}(q)]$, $[{}^2 A_{2d}(r)]$, $[^2 A_{2d-1}(r)]$, respectively.
(For details about the association schemes of dual polar spaces, see \cite{bcn,St80}.)
The dual polar scheme is accompanied by a regular semilattice \cite{CST08,D2}.
We can compute the structure constants using the parameters of these partially ordered sets.
Let $\{C_i\}_{i=0}^d$ be another basis ({\it Riemann basis} \cite{D2}) of $R\sX$ with 
$$
C_i=\sum_{j=0}^d \qbinom{j}{i} A_{d-j} \text{ for $i \in \{0,\dots,d\},$}
$$
where $\qbinom{a}{b}$ is a Gaussian coefficient.
We consider various homomorphisms between modular adjacency algebras of dual polar schemes 
using structure constants of $\{C_i\}_{i=0}^d$.
The structure constants $\{\rho_{s,t}^u\}_{s,t,u=0}^d$ satisfy
$$
C_s C_t = \sum_{u=0}^{\min(s,t)} \rho_{s,t}^u C_u.
$$
We define the $(d +1)\times (d+1)$ lower triangular matrix $G_s$ as
\begin{equation*}
(G_s)_{t,u} = \rho_{s,t}^u
\end{equation*}
and $RG=\oplus_{i=0}^{d} RG_i$ as a matrix ring over $R$.
This is a commutative algebra over $R$ and there is an isormorphism 
$\varphi : R\sX \longrightarrow RG$ given by $C_s \mapsto G_s$.

We can calculate the structure constants as follows.
\begin{lem}\label{StructureConstants}
Let $\rho_{s,t}^u$, $s,t,u \in \{0,1,\dots,d\}$, be the structure constant of $R \sX$.
Then
$$
\rho_{s,t}^u=\qbinom{d-u}{s-u}\qbinom{d-s}{t-u} q^{u(d-s-t+u)} (-q^{e-u};q)_{d-s-t+u},
$$
where $(a;q)_k$ is a $q$-Pochhammer symbol,
e is $1$, $1$, $0$, $2$, $3/2$, $1/2$ in the respective cases
$[C_n(q)]$, $[B_n(q)]$, $[D_n(q)]$, $[{}^2 D_{n+1}(q)]$, $[{}^2 A_{2n}(r)]$, $[{}^2 A_{2n-1}(r)]$ and $q=r^2$ for $[{}^2 A_{2n}(r)]$, $[{}^2A_{2n-1} (r)].$ 
\end{lem}

To prove Lemma \ref{StructureConstants}, we use the following properties of the Gaussian coefficients; see Proposition 8.5.2, Exercise 8.5.7 in \cite{CST08} and \cite{C08}.
\begin{prop}\label{qBinomialProperties}
\begin{enumerate}
\item \label{qBinomialProperty1}
\[
\binom{m}{n}_q \binom{n}{k}_q = \binom{m}{k}_q \binom{m-k}{n-k}_q;
\]
\item\label{qBinomialProperty2}
\[
\qbinom{m+n}{k}=\sum_{i+j=k}^{} \qbinom{m}{i}\qbinom{n}{j} q^{i(n-j)};
\]
\item\label{qBinomialProperty3}
\[
\sum_{i=0}^{m} (-1)^{m-i} q^{\binom{m-i}{2}} \qbinom{m}{i}\qbinom{i}{n}=\delta_{m,n};
\]
\item\label{qBinomialProperty4}
\[
\sum_{i=0}^{m} q^{\binom{i}{2}} \qbinom{m}{i} z^i = (-z;q)_{m} = (1+z)(1+zq)\dots(1+zq^{m-1})
=\prod_{l=0}^{m-1} (1+z q^l).
\]
\end{enumerate}
\end{prop}

\begin{proof}[Proof of Lemma \ref{StructureConstants}]
We use the same parameters of regular semilattice that we used in \cite{CST08},
and each parameters are  calculated:
\begin{equation}\label{eqrhorstmunupi}
\rho_{s,t}^u=\mu(s,u)\left[\sum_{j=0}^{u} \nu'(j,u) \pi(j,s,t)\right],
\end{equation}
where
\begin{align*}
&\mu(u,s)=\qbinom{d-u}{d-s}=\qbinom{d-u}{s-u} \text{ for $u\le s$ },\\
&\nu'(j,u)=(-1)^{u-j} q^{\binom{u-j}{2}} \binom{u}{j}_q,\\
&\text{ and }\\
&\pi(j,s,t)=\sum_{i=0}^{d}
\binom{d-s}{d-i}_q
q^{e(d-i)+\binom{d-i}{2}}
\binom{i-s+j}{t}_q \text{ for $s-j \le d-t, ~j \le s$}.
\end{align*}
The summation of (\ref{eqrhorstmunupi}) implies the following equation:
\begin{equation*}
\sum_{j=0}^{u} \nu'(j,u) \pi(j,s,t)
=\sum_{j=0}^{u} (-1)^{u-j} q^{\binom{u-j}{2}} \qbinom{u}{j} \sum_{i=0}^{d} 
\qbinom{d-s}{d-i} q^{e(d-i)+\binom{d-i}{2}} \qbinom{i-s+j}{t}.
\end{equation*}
By using  (\ref{qBinomialProperty2}) and (\ref{qBinomialProperty3})
in Proposition \ref{qBinomialProperties},
\begin{align*}
&\sum_{j=0}^{u} \nu'(j,u) \pi(j,s,t) \\
\qquad &=\sum_{i=0}^{d} \qbinom{d-s}{d-i} q^{e(d-i)+\binom{d-i}{2}}
\sum_{j=0}^{u} (-1)^{u-j} q^{\binom{u-j}{2}} \qbinom{u}{j} 
 \qbinom{i-s+j}{t}\\
\qquad  & =\sum_{i=0}^{d} \qbinom{d-s}{d-i} q^{e(d-i)+\binom{d-i}{2}}
\sum_{j=0}^{u} (-1)^{u-j} q^{\binom{u-j}{2}} \qbinom{u}{j} 
\sum_{h} \qbinom{i-s}{t-h}\qbinom{j}{h} q^{h(i-s-t+h)}\\
\qquad &=\sum_{i=0}^{d} \qbinom{d-s}{d-i} q^{e(d-i)+\binom{d-i}{2}}
\sum_{h}  \qbinom{i-s}{t-h} q^{h(i-s-t+h)} \delta_{u,h}\\
\qquad  & =\sum_{i=0}^{d} \qbinom{d-s}{d-i} q^{e(d-i)+\binom{d-i}{2}}
\qbinom{i-s}{t-u} q^{t(i-s-t+u)} \\ 
\qquad & =\sum_{i=0}^{d} q^{e(d-i)+\binom{d-i}{2} +t(i-s-t+u)}
\qbinom{d-s}{d-i} \qbinom{i-s}{t-u}.\\ 
\end{align*}
Hence, using (\ref{qBinomialProperty1}) in Proposition \ref{qBinomialProperties},
we have
$$
\sum_{j=0}^{u} \nu'(j,u) \pi(j,s,t) 
=\qbinom{d-s}{t-u}\sum_{i=0}^{d} q^{e(d-i)+\binom{d-i}{2} +t(i-s-t+u)}
 \qbinom{d-s-t+u}{i-s-t+u}.
$$
Also,
by change of a variable $d-i$ to $h$ and using (\ref{qBinomialProperty4}) in Proposition \ref{qBinomialProperties},
\begin{align*}
\rho_{s,t}^u 
&=\qbinom{d-u}{s-u}
\qbinom{d-s}{t-u}\sum_h q^{eh+\binom{h}{2} +u(d-h-s-t+u)}
 \qbinom{d-s-t+u}{h}\\
&=\qbinom{d-u}{s-u}\qbinom{d-s}{t-u} q^{u(d-s-t+u)} 
\sum_{h} q^{\binom{h}{2}} \qbinom{d-s-t+u}{h} (q^{e-u})^h\\
&=\qbinom{d-u}{r-u}\qbinom{d-s}{t-u} q^{u(d-s-t+u)} 
(-q^{e-u};q)_{d-s-t+u}.
\end{align*}
\end{proof}

\section{Locality of the modular adjacency algebra of a dual polar scheme}
We consider the decomposition of $F \sX$ to indecomposable two-sided ideals ({\it blocks}) \cite{NT87}.
Let $k(F \sX)$ be the number of blocks of $F \sX$.
Since the value of $k(F \sX)$ is equal to the number of $i \in \{0,1,\dots,d\}$ such that 
$p \nmid \rho_{i,i}^i$ (see \cite{S07}),
we can calculate it:
\begin{cor}[\cite{S07}]\label{NumOfBlocks}
$$
k(F\sX)=\left\{ i \in \{0,\dots,d\} \mid p \nmid \prod_{l=0}^{d-i-1} (q^i+q^{e+l})\right\},
$$
where $e$ is $1$, $1$, $0$, $2$, $3/2$, $1/2$ in the respective cases $[C_d(q)]$, $[B_d(q)]$, $[D_d(q)]$, $[{}^2 D_{d+1}(q)]$, $[{}^2 A_{2d}(r)]$, $[{}^2 A_{2d-1}(r)]$ and $q=r^2$ for $[{}^2 A_{2d}(r)]$, $[{}^2A_{2d-1} (r)]$.
\end{cor}

When the number of blocks of $F\sX$ is only one, $F \sX$ is a local algebra.
By Corollary  \ref{NumOfBlocks}, we know the following facts:

\begin{cor}\label{ConditionOfLocal}
Let $F$ be a field of odd characteristic $p$, $q$ be an odd prime power,
$\sX$ be  one of the dual polar scheme on $[C_d(q)]$, $[B_d(q)]$, $[D_d(q)]$, $[{}^2 D_{d+1}(q)]$, $[{}^2 A_{2d}(r)]$, $[{}^2 A_{2d-1}(r)]$ and $q=r^2$ for $[{}^2 A_{2d}(r)]$, $[{}^2A_{2d-1} (r)]$.

If $q$ is a prime power of $p$, then $F \sX$ is not a local algebra.

If $q$ is not a prime power of $p$, then 
$F \sX$ is a local algebra if and only if 
\begin{enumerate}
\item $q \equiv -1 \pmod{p}$ and $d$ is an odd for dual polar schemes on $[C_d(q)]$, $[B_d(q)]$, 
\item $q \equiv -1 \pmod{p}$ and $d$ is an even for dual polar schemes on $[D_d(q)]$, $[{}^2 D_{d+1}(q)]$,
\item $r \equiv -1 \pmod{p}$ for dual polar schemes on $[{}^2 A_{2d}(r)]$, $[{}^2 A_{2d-1}(r)]$.
\end{enumerate}
\end{cor}

\begin{proof}
We assume that $q$ is a prime power of $p$.
Since $\rho_{0,0}^{0} \equiv 1 \pmod{p}$ and the value of $\rho_{d,d}^d$ is one,
there is at least 2 blocks of $F \sX$. 
Thus $F \sX$ is not a local algebra.

We prove only the first case in the latter part.
We assume that $F \sX$ is a modular adjacency algebra on $[C_d(q)]$ or $[B_d(q)]$ of class $d$ and  $q$ is a prime power which is not divided by $p$.
If $F \sX$ is a local algebra, $p$ has to divide $\rho_{i,i}^i$ for any $i \in \{0,1,\dots,d-1\}$.
Since $p$ divides
\[ \rho_{0,0}^0 = \prod_{l=0}^{d-1}(1+q^l),\]
\[\rho_{d-2,d-2}^{d-2}=q(1+q^{d-2})(1+q^{d-3})\] 
and 
\[ \rho_{d-1,d-1}^{d-1} = 1+q^{d-1}, \] 
this lead to $q \equiv -1 \pmod{p}$ or $q \equiv 1 \pmod{p}$.
The case of $q \equiv 1 \pmod{p}$ is  a contradiction  with the fact that $p$ is an odd prime.
Thus, we have $q \equiv -1 \pmod {p}$.
Now we suppose that $d$ is an odd number.
Since $q^2 \equiv 1 \pmod{p}$,
$$\rho_{d-1,d-1}^{d-1} = (q^{d-1}+1) \equiv  2 \not\equiv 0 \pmod{p} .$$
Hence $d$ must be an even number.
To prove the sufficient condition assume that  $F \sX$ is a modular adjacency algebra on $[C_d(q)]$ or $[B_d(q)]$ of class  $d$ which is an even number and $q$ is a prime power such that $q +1\equiv 0 \pmod{p}$.
Since $p$ divides $\rho_{i,i}^i$ for any $i \in \{0,1,\dots,d-1\}$, 
$F \sX$ is a local algebra.  
\end{proof}
\begin{remark}
If $p$ divides $q$, then $F \sX$ is not a local algebra.
However, we have known the structure  of $F \sX$ for  $[C_n(q)]$, $[B_n(q)]$, $[{}^2 D_{n+1}(q)]$, $[{}^2 A_{2n}(r)]$, $[{}^2 A_{2n-1}(r)]$ in our results \cite{SY15}: 
$$
F \sX \cong F \oplus F[X]/(X^d).
$$
\end{remark}

A parameter $e$ of dual polar scheme  is   $1$, $1$, $0$, $2$, $3/2$, $1/2$ in the respective cases $[C_d(q)]$, $[B_d(q)]$, $[D_d(q)]$, $[{}^2 D_{d+1}(q)]$, $[{}^2 A_{2d}(r)]$.
Since a dual polar scheme on $[C_d(q)]$  has the same parameter $e$ as a dual polar scheme on $[B_d(q)]$, they are (algebraically) isomorphic.
We consider an (algebraically) isomorphism between dual polar schemes on different types except them.
Let $\sX_{q,d}^{e}$ be one of the dual polar scheme on  $[C_d(q)]$, $[B_d(q)]$, $[D_d(q)]$, $[{}^2 D_{d+1}(q)]$, $[{}^2 A_{2d}(r)]$, $[{}^2 A_{2d-1}(r)]$ and $q=r^2$ for $[{}^2 A_{2d}(r)]$, $[{}^2A_{2d-1} (r)]$ with parameter $e$.
\begin{prop}\label{IsomorphicCondition}
Let $F$ be  a field of characteristic $p$, $q$ be a prime power such that $p \nmid q_1,q_2$ and
If $q_1 \equiv q_2 \pmod{p}$ and 
$q_1^{e_1} \equiv q_2^{e_2} \pmod{p}$.
Then there is an isomorphism:
$$
F \sX_{q_1,d}^{e_1} \cong F\sX_{q_2,d}^{e_2},
$$
given by $C_r \mapsto \hat{C_r}$ ($0 \le r \le d$) where $\hat{C_r}$ is the $r$-th Riemann base of $F \sX_{q_2,d}^{e_2}.$
\end{prop}
\begin{proof}
If $q_1 \equiv q_2 \pmod{p}$ and ${q_1}^{e_1} \equiv {q_2}^{e_2} \pmod{p}$, then
$$
\binom{a}{b}_{q_1} \equiv \binom{a}{b}_{q_2} \pmod{p}
$$
and 
the values of $q$-Pochhammer symbols are congruent modulo $p$:
$$
(-{q_1}^{e_1-u};{q_1})_{d-s-t+u-1} \equiv (-{q_2}^{e_2-u};{q_2})_{d-s-t+u-1}  \pmod{p}
$$
for $s,t,u \in \{0,1,\dots,d\}.$
The above congruences and Lemma $\ref{StructureConstants}$ follow that
$$
\rho_{s,t}^u \equiv \hat{\rho}_{s,t}^u \pmod{p} \text{ for $s,t,u \in \{0,1,\dots,d \},$}
$$
where $\hat{\rho}_{s,t}^u$ is a structure constant of $F \sX_{q_2,d}^{e_2}.$
This implies that they are isomorphic.
\end{proof}
If the modular adjacency algebras of dual polar schemes on  $[D_d(q)]$ and $[{}^2 D_{d+1}(q)]$ ($[{}^2 A_{2d}(r)]$ and $[{}^2 A_{2d-1}(r)]$, respectively) are local algebra, then $q \equiv -1 \pmod{p}$ ($r \equiv -1 \pmod{p}$, respectively) by Corllary \ref{ConditionOfLocal}.
Since $q^2 \equiv 1 \pmod{p}$ ($q=r^2 \equiv 1 \pmod{p}$, respectivly),
$$
q^0 \equiv q^2 \pmod{p}
\text{
$(q^{1/2} \equiv q^{3/2} \pmod{p}$, respectively).}$$
By Proposition \ref{IsomorphicCondition}, we conclude the following corollary.
\begin{cor}\label{IsomorphicA}
Let $F$ be a field of characteristic $p$, $q$ and $r$ be a prime power which is not divided by $p$.
If dual polar schemes on  $[D_d(q)]$ and $[{}^2 D_{d+1}(q)]$ ($[{}^2 A_{2d}(r)]$ and $[{}^2 A_{2d-1}(r)]$, respectively) are local algebras,
then they are isomorphic.
\end{cor}

\section{The structure of the modular adjacency algebra of a dual polar scheme when it is local}
In the remaining part of this paper we consider the structure of the modular adjacency algebra of a dual polar scheme when it is a local algebra.
Therefore, we assume that $p$ is an odd prime and the following conditions hold:
\begin{enumerate}
\item $q \equiv -1 \pmod{p}$ and $d$ is an odd number for dual polar schemes on $[C_d(q)]$, $[B_d(q)]$, 
\item $q \equiv -1 \pmod{p}$ and $d$ is an even number for dual polar schemes on $[D_d(q)]$, $[{}^2 D_{d+1}(q)]$,
\item $r \equiv -1 \pmod{p}$ for dual polar schemes on $[{}^2 A_{2d}(r)]$, $[{}^2 A_{2d-1}(r)]$.
\end{enumerate}
The next proposition decides the structure of the modular adjacency algebra of a dual polar scheme when it is a local algebra.
\begin{prop}\label{MatrixOfLocal}
Let $F$ be a field of odd characteristic $p$, $q$ be an odd prime power,
$\sX$ be a dual polar scheme.
If $F \sX$ be a local algebra, then
$$
\rho_{s,t}^u \equiv 
\begin{cases}
\qbinom{d-u}{s-u} \pmod{p} & \text{  if $d-s=t-u$ },\\
0 \pmod{p} & \text{ otherwise }.
\end{cases}
$$
\end{prop}
\begin{proof}
We only prove for a dual polar scheme on $[C_d(q)]$; proofs of other types are similar. 
Since we assume that  $F \sX$ is a local algebra, put $q \equiv -1 \pmod{p}$ and $d$ is an even number by Corollary \ref{ConditionOfLocal}.
Using a combinatorial interpretation of the Gaussian coefficient as a polynomial in $q$ by Theorem 24.2 in \cite{V01},
$$
\qbinom{n}{k}= \sum_{l=0}^{k(n-k)} a_l q^l,
$$
where the coefficient $a_l$ is the number of partitions of $l$ whose Ferrers diagrams fit in a box of size $k \times (n-k).$
The following equation is due to a complement of each Ferrers diagram in the box.
$$a_{l}=a_{k(n-k)-l}.$$
This yields that $\qbinom{n}{k} \equiv 0 \pmod{p}$ when $k(n-k)$ is an odd number
because we assume that $q \equiv -1 \pmod{p}$.
Hence it is enough to consider the case of 
$(t-u)\{(d-s)-(t-u)\}$ and $(d-s)(s-u)$ are even.
Furthermore,
we can assume that 
$d-s-t+u $ is $0$ or $1$. 
It is caused by functors of the structure constant: 
$$
q^{u(d-s-t+u)} (-q^{e-u};q)_{d-s-t+u} = \prod_{l=0}^{d-s-t+u-1} (q^u+q^{1+l}).
$$
When the parity of exponents of $q$ is different,
$q^u + q^{1+l} \equiv 0 \pmod{p}.$
If $d-s-t+u = 1$, then
our assumptions lead to $u$ is an even number.
This means that
$$q^u+q^1 \equiv 0 \pmod{p}.$$
Hence  the value of $\rho_{s,t}^u$ is  congruent to $0$ except the case of $d-s-t+u=0.$
\end{proof}

We need to proof some properties for dual polar schemes on $[^2 A_{2d}(r)]$ before we consider 
Theorem \ref{MainTheorem1}.

\begin{cor}\label{EpiOfA}
Let $F$ be a field of odd characteristic $p$, $r$ be an odd prime power which is not divided by $p$ and
$\sA_d$ be a dual polar scheme on $[^2 A_{2d}(r)]$.
If $F \sA_{d}$ is a local algebra,
then there is an epimorphism such that
$$
\varphi_{d+1} : F \sA_{d+1} \longrightarrow F \sA_{d},
$$
given by $C_{s} \mapsto \hat{C}_{s-1}$ ($0 \le s\le d+1$) where $\hat{C}_{s-1}$ is the $(s-1)$-th Riemann base of $F \sA_{d}$ and $\hat{C}_{-1} = 0.$
\end{cor}
\begin{proof}
By Proposition \ref{MatrixOfLocal}, 
we have a congruence:
$$
\rho_{s,t}^{u,d+1}\equiv\rho_{s-1,t-1}^{u-1,d}\pmod{p} \text{ for $s,t,u \in \{1,2,\dots,d+1\}$,}
$$
where $\rho_{s,t}^{u,d}$ is a structure constant of $F \sA_{d}.$
This congruence leads to the map $\varphi_{d+1}$ being an algebraic homomorphism.
Thus the map $\varphi_{d+1}$ is an epimorphism since it is clearly surjective.
\end{proof}

\begin{prop}\label{TensorDecomposition}
Let $F$ be a field of odd characteristic $p$, $r$ be an odd prime power which is not divided by $p$ and
$\sA_d$ be a dual polar scheme on $[^2 A_{2d}(r)].$
If $F \sA_{p^l-1}$ is a local algebra,
then
$$
F \sA_{p^l-1} \cong \bigotimes^l F \sA_{p-1}.
$$
\end{prop}
\begin{proof}
To prove statement,
we show that 
$$
C_{p s+\alpha}^{p^l-1} C_{p t+\beta}^{p^l-1} 
=
\left( \sum_{u=0}^{p^{l-1}-1} \sum_{\gamma}^{p-1} \rho_{s,t}^{u,p^{l-1}} \rho_{\alpha,\beta}^{\gamma,p-1} \right) 
C_{p u +\gamma}^{p^l-1}
$$
for $0 \le s,t \le p^{l-1}$ and $0 \le \alpha, \beta \le p-1$,
where $C_{s}^{d}$ is an $s$-th Riemann base of $F \sA_d$ and $\rho_{s,t}^{u,d}$ is a structure constant of $F \sA_d$. 
This leads to
$$
C_{ps+\alpha}^{p^l-1} = C_{s}^{p^{l-1}-1} \otimes C_{\alpha}^{p-1}.
$$
By repeating the same argument, each Riemann base of $F\sA_{p^l-1}$ can be decomposed
as $l$ tensor products of Riemann basis of $F\sA_{p-1}.$
This implies that our statement.
Thus it suffices to verify that 
$$
\rho_{ps+\alpha, pt+\beta}^{pu+\gamma,p^l-1} \equiv \rho_{s,t}^{u,p^{l-1}-1} \rho_{\alpha,\beta}^{\gamma,p-1}
\pmod{p}
$$
for $0 \le s,t,u \le p^{l-1}-1$, $0 \le \alpha,\beta,\gamma \le p-1.$
Now we assume that $q=r^2 \equiv 1 \pmod{p}.$
This means that any Gaussian coefficient  is congruent to a binomial coefficient.
By Proposition \ref{MatrixOfLocal} and Lucas' theorem \cite{S88},
\begin{align*}
\rho_{ps+\alpha, pt+\beta}^{pu+\gamma,p^l-1} 
&= \qbinom{p^l-1-(pu+\gamma)}{(ps+\alpha)-(pu+\gamma)}\\
&\equiv \binom{p^l-1-(pu+\gamma)}{(ps+\alpha)-(pu+\gamma)}\\
&\equiv \binom{p^{l-1}-1-u}{s-u} \binom{p-1-\gamma}{ \alpha-\gamma}\\
&\equiv \qbinom{p^{l-1}-1-u}{s-u} \qbinom{p-1-\gamma}{ \alpha-\gamma}\\
& = \rho_{s,t }^{u,p^{l-1}-1} \rho_{\alpha,\beta}^{\gamma,p-1}
\end{align*}
This completes the proof.
\end{proof}

Since the lower triangular matrix $G_{p-2}$ of $F \sA_{p-1}$ is nilpotent and its rank is $p-1$,
we determine the structure of $\sA_{p-1}.$
\begin{prop}\label{StructureSmallOne}
Let $F$ be a field of odd characteristic $p$, $r$ be an odd prime power which is not divided by $p$ and $\sA_{d}$ be a dual polar scheme on $[^2A_{2d}(r)]$.
If $F \sA_{p-1}$ is a local algebra, then there is an isomorphism such that
$$
F\sA_{p-1} \cong F[X]/(X^p),$$
given by $G_{p-2} \mapsto X.$
\end{prop}
By Proposition \ref{TensorDecomposition} and \ref{StructureSmallOne},
we know that the modular adjacency algebra $F \sA_{p^l-1}$ is isomorphic to the group algebra of the elementary abelian group of order $p^l$ over $F$.
The group algebra is isomorphic to a quotient ring $F[X_1,X_2,\dots,X_l]/(X_1^p,X_2^p,\dots,X_l^p)$ (e.g., Chapter V Lemma 5.3 in \cite{EORT}).
We put $P_{l}=F[X_1,X_2,\dots,X_l]/(X_1^p,X_2^p,\dots,X_l^p).$
\begin{thm}\label{StructureOfPowerOfP}
Let $F$ be a field of odd characteristic $p$, $r$ be an odd prime power which is not divided by $p$ and
$\sA_d$ be a dual polar scheme on $[^2 A_{2d}(r)]$.
If $F \sA_{p^l-1}$ is a local algebra,
then 
$$F \sA_{p^l-1} \cong P_l.$$
\end{thm}
Next we determine the structure of $F \sA_{d}.$
To do this, we need some properties of the Riemann basis of $F \sA_{p^l-1}.$

\begin{lem}\label{PropertiesOfA}
Let $F$ be a field of odd characteristic $p$, $r$ be an odd prime power which is not divided by $p$ and
$\sA_d$ be a dual polar scheme on $[^2 A_{2d}(r)]$.
If $F \sA_{p^l-1}$ is a local algebra,
then the following congruences are hold:
\begin{enumerate}
\item
$$
(C_{p^l-1-p^i})^m \equiv m! C_{p^l-1-mp^i} \pmod{p} \text{ for $m \in \{1,2,\dots,p-1\};$}
$$
\item  
$$
C_{p^l-1-\alpha p^i} C_{p^l-1- \beta p^j} \equiv C_{p^l-1-(\alpha p^i + \beta p^j)}\pmod{p}  
\text{ for $i \ne j$ and $\alpha,\beta \in \{0,1,\dots,p-1\}.$ }
$$
\end{enumerate}
\end{lem}
\begin{proof}
We prove the first congruence by induction on $m$:
Assuming the congruence holds for  $m-1$ we need to prove that
the congruence holds for $m$.
\begin{align*}
\left( C_{p^l-1-p^i} \right)^m &=(m-1)!C_{p^l-1-(m-1)p^i} C_{p^l-1-p^i}\\
&=(m-1)! \sum_{k=0}^{p^l-1} \rho_{p^l-1-(m-1)p^i,p^l-1-p^i}^k C_k\\
&=(m-1)!  \rho_{p^l-1-(m-1)p^i,p^l-1-p^i}^{p^l-1-mp^i} C_{p^l-1-mp^i}.
\end{align*}
By Proposition \ref{MatrixOfLocal},
terms are eliminated except for $k=p^l-1-mp^i.$
Since we assume that $q \equiv r^2 \equiv 1 \pmod{p},$
any Gaussian coefficient is congruent to a binomial coefficient.
Therefore the following congruence holds by Lucas' Theorem.
\begin{align*}
\rho_{p^l-1-(m-1)p^i,p^l-1-p^i}^{p^l-1-mp^i} 
&\equiv  \qbinom{mp^i}{p^i} \\
&\equiv  \binom{mp^i}{p^i} \\
&\equiv m \pmod{p}.
\end{align*}
Similarly, we prove that the the second congruence by Lucas' Theorem.
\end{proof}
By the above Lemma \ref{PropertiesOfA},
for an isomorphism of Theorem \ref{StructureOfPowerOfP}:
$$
s_l : P_l \longrightarrow \sA_{p^l-1},
$$
we can define the correspondence
$$
s_l(X_i^k)=(C_{p^l-1-p^{i-1}})^k.
$$

At last, we prove our main theorems. 
\begin{proof}[proof of Theorem \ref{MainTheorem1}]
We define a positive integer weight function $wt$ over $P_l$,
say $wt:P_l \longrightarrow \mathbb{Z}^+,$ such that
$$
wt(X_i) = p^{i-1}
\text{ and }
wt(X_{i_0}^{k_0}X_{i_1}^{k_1}\cdots X_{i_s}^{k_s} )=\sum_{j=0}^{s} k_j p^{i_j-1}.
$$
Put $Y_d = X_{i_0}^{k_0}X_{i_1}^{k_1}\cdots X_{i_s}^{k_s} \in P_l$ such that $wt(Y_d)=d~(\le p^l-1)$.
We can prove that
$$
Y_d \in Ker \varphi_d \varphi_{d-1} \cdots \varphi_{p^l-1} s_l \text{ for $d \in \mathbb{Z}^+$ }
$$
using the same argument as the proof of Theorem 8 in \cite{Y04} by Lemma \ref{PropertiesOfA}.
\end{proof}

Next we prove our second main theorem.
\begin{proof}[Proof of Theorem \ref{MainTheorem2}]
Let $F$ be a field of odd characteristic $p$, $q$ be an odd prime power which is not divided by $p$ and
$\sC_d$ be a dual polar scheme on $[C_{d}(q)]$.
If $F \sC_{d}$ is a local algebra,
$d$ must be an odd number
by Corollary \ref{ConditionOfLocal}.
We put $d=2d'+1$.
Since $q \equiv -1 \pmod{p}$,
a modular adjacency algebra $F\sA_{d}$ of a dual polar scheme on $[^2 A_{2d} (q)]$ is a local algebra. 
We prove that there is an isomorphism:
$$
F \sC_{2d'+1} \cong F \sA_{d'} \otimes F \sA_{1},
$$
given by $C_{2s+\alpha}^{2d'+1} \mapsto \hat{C}_{s}^{d'} \otimes \hat{C}_{\alpha}^1\text{ ( $0 \le s \le d', 0\le \alpha \le 1$) },$
where $\hat{C}_{s}^{d'}$ is a $s$-th Riemann base of $F \sA_{d'}.$
To do this,
we  prove 
$$
\rho_{2s+\alpha,2t+\beta}^{2u+\gamma,2d'+1} \equiv 
\hat{\rho}_{s,t}^{u,d'} \hat{\rho}_{\alpha,\beta}^{\gamma,1} \pmod{p}
\text{ for $s,t,u \in \{0,1,\dots, d'\}$, $\alpha,\beta,\gamma \in \{0,1\},$}
$$
where $\rho_{s,t}^{u,d}$ ($\hat{\rho}_{s,t}^{u,d}$)  is a structure constant of $F \sC_d$ ($F \sA_{d}$, respectively).
By Proposition \ref{MatrixOfLocal},
if $\rho_{2s+\alpha,2t+\beta}^{2u+\gamma,2d'+1} \not\equiv 0$, then
$$
\rho_{2s+\alpha,2t+\beta}^{2u+\gamma,2d'+1}  \equiv \qbinom{2d'+1-(2u+\gamma)}{2s+\alpha-(2u+\gamma)}
\equiv \qbinom{2(d'-u)+(1-\gamma)}{2(s-u)+(\alpha-\gamma)}.
$$
By a congruence of Gaussian coefficients and binomial coefficients \cite{S88},
\begin{equation}\label{congrhoOfBig}
\qbinom{2(d'-u)+(1-\gamma)}{2(s-u)+(\alpha-\gamma)} \equiv \binom{d'-u}{s-u}\qbinom{1-\gamma}{\alpha-\gamma} \pmod{p}.
\end{equation}
On the otherhand,
since $q^2 \equiv 1 \pmod{p}$,
\begin{equation}\label{congrhoOfSmall}
\hat{\rho}_{s, t}^{u, d'} \hat{\rho}_{\alpha, \beta}^{\gamma, 1} \equiv 
\qqbinom{d'-u}{s-u} \qqbinom{1-\gamma}{\alpha-\gamma} \equiv
\binom{d'-u}{s-u} \binom{1-\gamma}{\alpha-\gamma} \pmod{p}.
\end{equation}
The values of the second factors of two congruences (\ref{congrhoOfBig}) and (\ref{congrhoOfSmall}) are one.
This completes the proof of Theorem \ref{MainTheorem2}.
\end{proof}
The following proposition determines the structure of modular adjacency algebras of dual polar schemes on remaining types $[D_d(q)]$ and $[^2D_{d+1}(q)]$.
 We can prove this proposition like the proof of  Corollary \ref{EpiOfA}. 
 \begin{prop}
 Let $F$ be a field of odd characteristic $p$, $q$ be an odd prime power which is not divided by $p$ and
$\sC_d$ ($\sD_d$) be a dual polar scheme on $[C_{d}(q)]$ ($[D_d(q)]$, respectively).
If $F \sC_{d}$ is a local algebra,
then there is an epimorphism
$$
\psi_{d}: F \sC_{d} \longrightarrow F\sD_{d-1},
$$
given by $C_s \mapsto \hat{C}_{s-1}$ ($0 \le s \le d$)
where $\hat{C}_{s-1}$ is the $(s-1)$-th  Riemann base of $F \sD_d$ and $\hat{C}_{-1} = 0.$
 \end{prop}
 
 Since $Ker \psi_{2d'+1}=C_0 \in F\sC_{2d'+1}$, it correspond to the tensor products of the highest weight polynomials in $P/W_{d'}$ and $P/W_1$ by Theorem \ref{MainTheorem1}.
Therefore
$$
F \sD_{2d'} \cong (P/W_{d'} \otimes P/W_{1})/(\overline{Y_{d'}} \otimes  \overline{Y_1}).
$$
\begin{remark}
If $p$ is the even prime number, then $F \sX$ is a local algebra where $\sX$ be  one of the dual polar scheme on $[C_d(q)]$, $[B_d(q)]$, $[D_d(q)]$, $[{}^2 D_{d+1}(q)]$, $[{}^2 A_{2d}(r)]$, $[{}^2 A_{2d-1}(r)]$ and $q=r^2$ for $[{}^2 A_{2d}(r)]$, $[{}^2A_{2d-1} (r)]$ for any odd prime power $q$ by Corollary \ref{NumOfBlocks}.
In particular, by Proposition \ref{IsomorphicCondition} and a same argument as in the proof of Theorem \ref{MainTheorem1}, $F \sX$ is isomorphic to $P/W_d.$
\end{remark}

\end{document}